\documentclass[10pt,twocolumn,twoside,a4paper]{article}

\usepackage[left=15mm, right=15mm, top=15mm, bottom=15mm]{geometry}

\usepackage{subfig}
\usepackage{flushend}
\usepackage{indentfirst}
\usepackage{graphics}
\usepackage{amsmath}
\usepackage{graphicx}
\usepackage{epstopdf}
\usepackage{float}

\usepackage[font=footnotesize,labelfont=bf]{caption}
\usepackage[labelsep=period]{caption}
\usepackage{fancyhdr}

\rmfamily

\graphicspath{{./figure/}}

\usepackage[latin1]{inputenc}
\usepackage{amsmath}
\usepackage{amsfonts}
\usepackage{amssymb}
\usepackage{graphicx}
\usepackage{amsthm}
\usepackage{tikz}
\usepackage{braids}
\usepackage{cite}
\usepackage{amssymb}

\newtheorem{theorem}{Theorem}[section]
\newtheorem{remark}[theorem]{Remark}

\newtheorem{example}[theorem]{Example}

\newcommand{\rnc}[2]{\renewcommand{#1}{#2}}
\rnc{\theequation}{\thesection.\arabic{equation}}
\rnc{\[}{\begin{equation}}

\fancypagestyle{plain}{
\fancyhead{}
\fancyhead[L]{Mathematics and Statistics 10(1): 145-152, 2022\\
		      DOI: 10.13189/ms.2022.100112}
\fancyhead[R]{http://www.hrpub.org\\}
}
\begin{document}

\title{\huge \textbf{A basic $n$ Dimensional Representation of Artin Braid Group $B_n$, and a  General Burau Representation}}

\twocolumn[
\begin{@twocolumnfalse}
\author{\textbf{Arash Pourkia}\\\\
\footnotesize {College of Engineering and Technology, American University of the Middle East, Kuwait }\\
%\footnotesize $^{2}$Chemistry and Chemical Engineering Faculty, California Institute of Technology, Pasadena, 91125, California, United States\\
%\footnotesize $^{3}$College of Arts and Sciences, University of Pennsylvania, Philadelphia, 19104, Pennsylvania, United States\\
\footnotesize $^{*}$Corresponding Author: arash.pourkia@aum.edu.kw}

\date{}

\maketitle

\textit{\footnotesize \centerline{Received September 17, 2021; Revised November 10, 2021; Accepted November 21, 2021}}

\noindent \textbf{\textit{\footnotesize Cite This Paper in the following Citation Styles}}\\
\textit{\footnotesize \textbf{(a)}: [1] Arash Pourkia, "A basic $n$ Dimensional Representation of Artin Braid Group $B_n$, and a  General Burau Representation," Mathematics and Statistics, Vol.10, No.1, pp. 145-152, 2022. DOI: 10.13189/ms.2022.100112 }\\
\textit{\footnotesize \textbf{(b)}: Arash Pourkia, (2022). A basic $n$ Dimensional Representation of Artin Braid Group $B_n$, and a  General Burau Representation. Mathematics and Statistics, 10(1), 145-152. DOI: 10.13189/ms.2022.100112 }\\
\begin{flushleft}
\noindent \footnotesize {Copyright \copyright 2022 by authors, all rights reserved. Authors agree that this article remains permanently open access under the terms of\\
\noindent \footnotesize  the Creative Commons Attribution License 4.0 International License}
\end{flushleft}

\end{@twocolumnfalse}
]

\noindent \textbf{\large{Abstract}} \hspace{2pt} Braid groups and their representations are at the center of study, not only in low-dimensional topology, but also in many other branches of mathematics and theoretical physics. Burau representation of the Artin braid group which has two versions, reduced and unreduced, has been the focus of extensive study and research since its discovery in 1930's. It remains as one of the very important representations for the braid group. Partly, because of its connections to the Alexander polynomial which is one of the first and most useful invariants for  knots and links. In the present work, we show that interesting representations of braid group could be achieved using a simple and intuitive approach, where we simply analyse the path of strands in a braid and encode the over-crossings, under-crossings or no-crossings into some parameters. More precisely, at each crossing, where, for example, the strand $i$ crosses over the strand $i+1$ we assign $\mathbf{t}$ to the {\bf t}op strand and $\mathbf{b}$ to the {\bf b}ottom strand. We consider the parameter $\mathbf{t}$ as a {\it relative weight} given to strand $i$ relative to $i+1$,  hence the position $i\ i+1$ for $\mathbf{t}$ in the matrix representation. Similarly, the parameter $\mathbf{b}$ is a {\it relative weight} given to strand $i+1$ relative to $i$,  hence the position $i+1\ i$ for $\mathbf{b}$ in the matrix representation. We show this simple {\it path analyzing approach} that leads us to an interesting simple representation. Next, we show that following the same intuitive approach, only by introducing an additional parameter, we can greatly improve the representation into the one with much smaller kernel. This more general representation includes the unreduced Burau representation, as a special case. Our new {\it path analyzing approach} has the advantage that it applies a very simple and intuitive method capturing the fundamental interactions of the strands in a braid. In this approach we intuitively follow each strand in a braid and create a {\it history} for the strand as it interacts with other strands via over-crossings, under-crossings or no-crossings. This, directly, leads us to the desired representations. \\

\noindent \textbf{\large{Keywords}} \hspace{2pt} Artin Braid Group, Braid Group Representations, Burau Representations\\

\noindent\hrulefill

%%%%%%%%%%%%%%%%%%%%%%%%%
\section{\Large{Introduction}}

Braid groups, introduced by Artin in the 1920, and their representations are among central fields of study in low-dimensional topology and many other branches of mathematics and theoretical physics \cite{kassel-turaev book, Birman-book, Birman-Brendle005}. Burau representation of the Artin braid group was discovered in 1930's, by Werner Burau \cite{Burau}. It is defined by considering the braid group $B_n$ as the mapping class group of a disc with $n$ marked points, and studying the first homology group of the disc. The Burau representation, for $n$ larger than $3$, except for $n=4$, was proved to be nonfaithful through the works of John A. Moody, Darren D. Long and M. Paton, and  Stephen Bigelow in 1990's \cite{Moody91, Long-Paton93, Bigelow99}. The faithfulness of the Burau representation for $n=4$ is still an open problem. However, the Artin braid group was proved to be linear, through the works of Ruth J.  Lawrence, Stephen Bigelow, and Daan Krammer during the 1990's and early 2000's \cite{Lawrence90, Krammer00,Krammer02,Bigelow01, Turaev00}. The Burau representation, which has two versions, namely, reduced and unreduced Burau representation, remains as a very important representation for the braid group. Partly, because the reduced version has relations to the Alexander polynomial which in turn provides link/knot invariant  \cite{kassel-turaev book, Birman-book, Birman-Brendle005, kauffman knot book, colin knot book}.

\subsection*{\normalsize \textbf{The motivation and summary of results}}

The motivation behind the current work is to explore what could be achieved if we intuitively follow the path of each strand in a braid and create a sort of {\it history} for the strand as it interacts with other strands via over-crossings, under-crossings or no-crossings. We call this new method the {\it path analyzing approach}. This approach, as we will see, has the advantage that it applies a very simple and intuitive method capturing the fundamental interactions of the strands in a braid, i.e., over-crossings, under-crossings or no-crossings. This approach, as we will briefly explain now, will hep us to define new representations of the braid group $B_n$, some of which are a general versions of previously known representations. \\

At each crossing, where the strand $i$ crosses over the strand $i+1$, for example, we assign $\mathbf{t}$ to the {\bf t}op strand and $\mathbf{b}$ to the {\bf b}ottom strand. We consider the parameter $\mathbf{t}$ as a {\it relative weight} given to strand $i$ relative to $i+1$,  hence the position $i\ i+1$ for $\mathbf{t}$ in the matrix. Similarly, the parameter $\mathbf{b}$ is a {\it relative weight} given to strand $i+1$ relative to $i$,  hence the position $i+1\ i$ for $\mathbf{b}$ in the matrix (Figure \ref{fig2}). This simple intuitive approach directly leads us to the very simpler representation of Theorem \ref{maintheorem} defined by,

\begin{equation*}\Psi(\sigma_i)=\begin{pmatrix}
I_{i-1}
& \vline &0 \cdots & \vline & 0 \cdots\\
\hline
0 \cdots & \vline &
\begin{matrix}
0 & t \\
b & 0
\end{matrix}
& \vline & 0 \cdots \\
\hline
0 \cdots & \vline & 0 \cdots & \vline & I_{n-i-1}\\
\end{pmatrix}, \end{equation*}
which is also {\bf Case 3} in Theorem \ref{maintheorem2}, after renaming the parameters.\\

Next, to improve this representation to the one with much smaller kernel we go one step further, as follows. Not only we assign $\mathbf{t}$ and $\mathbf{b}$ as described above, but also we assign to the top strand a new {\it weight}, $\mathbf{\omega}$, relative to itself (hence the  $i\ i$ position). After some inspection we realise the only way for the representation to work is if $\mathbf{\omega = 1-tb}$. This leads to the more general representation defined by,

\begin{equation*}\Psi(\sigma_i)=\begin{pmatrix}
I_{i-1}
& \vline &0 \cdots & \vline & 0 \cdots\\
\hline
0 \cdots & \vline &
\begin{matrix}
1-tb & t \\
b & 0
\end{matrix}
& \vline & 0 \cdots \\
\hline
0 \cdots & \vline & 0 \cdots & \vline & I_{n-i-1}\\
\end{pmatrix}, \end{equation*}
which, after renaming of the parameters, is {\bf Case 1} in Theorem \ref{maintheorem2}. Similarly, if we assign to the bottom strand the {\it weight} $\mathbf{1-tb}$, we will have

\begin{equation*}\Psi(\sigma_i)=\begin{pmatrix}
I_{i-1}
& \vline &0 \cdots & \vline & 0 \cdots\\
\hline
0 \cdots & \vline &
\begin{matrix}
0 & t \\
b & 1-tb
\end{matrix}
& \vline & 0 \cdots \\
\hline
0 \cdots & \vline & 0 \cdots & \vline & I_{n-i-1}\\
\end{pmatrix} ,\end{equation*}
which is {\bf Case 2} in Theorem \ref{maintheorem2}, after renaming of the parameters. In Theorem \ref{maintheorem2}, we also show that it is not possible to assign the new weight, $\mathbf{\omega}$, to both strands at the same time, as one might wonder to do so. \\

 The present paper is organized as follows. In Section \ref{core idea} we describe the core idea of our {\it path analyzing approach} and use it to define the basic representation. Then we show that this representation is not faithful. In Section \ref{General Burau representation}, continuing our {\it path analyzing approach} we introduce a more general representation for $B_n$, which has much more smaller kernel than the basic representation. It also includes, as a special case, the unreduced Burau representation. We finish this paper with  the concluding remarks in Section \ref{Concluding remarks}, followed by the acknowledgement.

%%%%%%%%%%%%%%%%%%%%%%%%%	
\section{\Large{The core idea: Path Analyzing Approach}} \label{core idea}

\subsection{\normalsize \textbf{Preliminaries and Notations}}
We denote by $B_n$ the $n$-strands braid group with Artin generators $\sigma_i,\, 1\leq i \leq n-1$, satisfying the so-called cubic and commuting relations,
\begin{equation} \label{cubic relations} \sigma_{i}\sigma_{i+1}\sigma_{i}=\sigma_{i+1}\sigma_{i}\sigma_{i+1} \,\, \text{for} \,\, 1\leq i\leq n-2 \, \, ,\end{equation}
\begin{equation}\label{commuting relations} \sigma_{i}\sigma_{j}=\sigma_{j}\sigma_{i} \,\, \text{for} \,\, |i-j|\geq 2 \end{equation}

For a given braid $\sigma$ in $B_n$, which is composed of powers of some $\sigma_i$'s, we use the common pictorial representation. We use two sets of points, both depicted by numbers $1, 2, \cdots n$ (from left to right), and we call them higher points and lower points. When there is no confusion we might simply say a higher $i$ and a lower $j$, or might even refer to them only as $i$ and $j$ (Figure \ref{fig1}). For our purpose, we fix the downward direction for all braids. This means we always imagine that any strand in a given braid, "like the path of a moving particle", starts at a higher point indicated by a higher number $i$, $1\leq i \leq n$ and travels {"all the way down"}, to arrive at a lower point indicated by a lower number $j$, $1\leq j \leq n$. In doing so, the strand might passes through some crossings of type $\sigma_i$'s and  $\sigma^{-1}_i$'s. We will call the crossing in a $\sigma_i$ a positive crossing or an over-crossing, and call the crossing in its inverse $\sigma_i^{-1}$ a negative crossing or an under-crossing (Figure \ref{fig1}).

\begin{figure}[thbp]
\centering
\includegraphics[width=7cm,height=4cm]{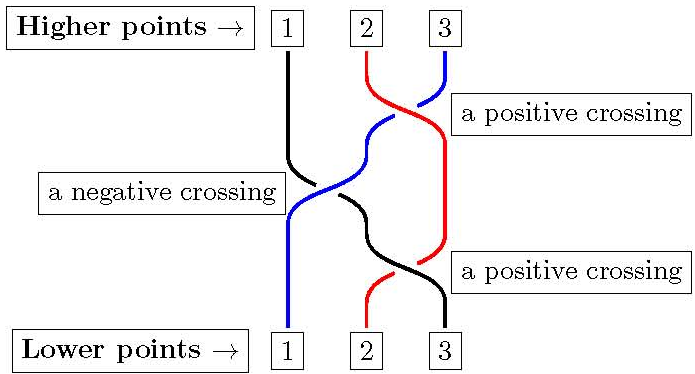}
\caption{Examples of positive and negative crossings.}
\label{fig1}
\end{figure}

We will call a strand, starting from a higher $i$ the $i^{th}$-strand. The $i^{th}$-strand moving downward will end up at a lower $j$. To such a strand, we will assign a value $\psi_{ij}$ {\it which is equal to the product of all values of crossings it has passed throughout its downward travel}. This way, to any given braid $\sigma$ in $B_n$, we will assign a $n$ by $n$ matrix! All these will become more clear in what follows.
	
\subsection{\normalsize \textbf{The core idea}}
The core idea of this paper is as follows. Let $B_2$ be the $2$-strands braid group with the generator $\sigma$. We will call $\sigma$ the positive crossing (also called the over-crossing) and call its inverse, $\sigma^{-1}$, the negative crossing (also called the under-crossing). Let $\mathbb{Z}[t^{\pm 1}, b^{\pm 1}]$ be the ring of Laurent polynomials in two variables, over $\mathbb{Z}$. We will assign two $2$ by $2$ matrices $\psi$ and $\psi^{-1}$ in $GL_2(\mathbb{Z}[t^{\pm1}, b^{\pm 1}])$ ($2$ dimensional representations), to $\sigma$ and to $\sigma^{-1}$ respectively, as follows (Figure \ref{fig2}).

	In the positive crossing $\sigma$, the $1^{st}$ strand, which is also the {\bf top} strand, goes downward from the higher point $1$ to the lower point $2$. We assign the entry $\psi_{12}=t$ to this strand. At the same time the $2^{nd}$ strand, which is also the {\bf bottom} strand, goes downward from the higher $2$ to the lower $1$. We assign the entry $\psi_{21}=b$ to this strand. In the negative crossing $\sigma^{-1}$, the top strand goes downward from the higher $2$ to the lower $1$. We assign the entry $\psi^{-1}_{21}=t^{-1}$ to this strand. At the same time the bottom strand goes downward from the higher $1$ to the lower $2$. We assign the entry $\psi^{-1}_{12}=b^{-1}$ to this strand. Moreover, since there is no strand connecting the higher $1$ to the lower $1$, we set $\psi_{11}=0$. For a similar reason $\psi_{22}=0$, and also $\psi^{-1}_{11}=\psi^{-1}_{22}=0$.
	
	In summary, for positive crossings, $\mathbf{t}$ is assigned to the  {\bf top} strand and $\mathbf{b}$ is assigned to the {\bf bottom} strand, at each crossing point. Similarly, for negative crossings, $\mathbf{t^{-1}}$ is assigned to the {\bf top} strand and $\mathbf{b^{-1}}$ is assigned to the {\bf bottom} strand, at each crossing point (Figure \ref{fig2}).  %\newpage
	
\begin{figure}[thbp]
\centering
\includegraphics[width=7cm,height=4cm]{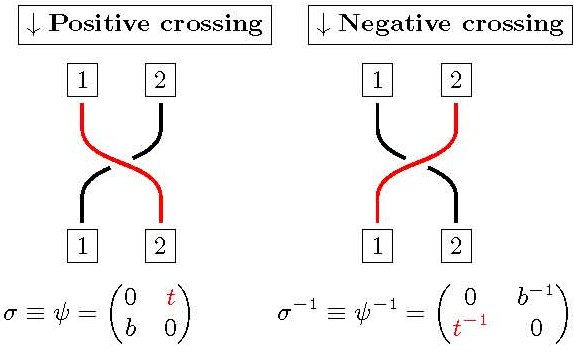}
\caption{Assigning matrices to basic positive and negative crossings.}
\label{fig2}
\end{figure}
	
To extend this idea to $B_n$ with generators $\sigma_1$, $\sigma_2$ ,etc., it is enough to to assign the value $1$ to any strand who goes downward from a higher point $i$ to a lower point $i$ straight away without being involved with any crossings. This means the $ii$-entry=$1$, for such a $i$.  Also remember that if there is no strand connecting a higher $i$ to a lower $j$ then the $ij$-entry=$0$. For $B_3$ this is shown in Figures \ref{fig3} and \ref{fig4}.

\begin{figure}[thbp]
\centering
\includegraphics[width=7cm,height=4cm]{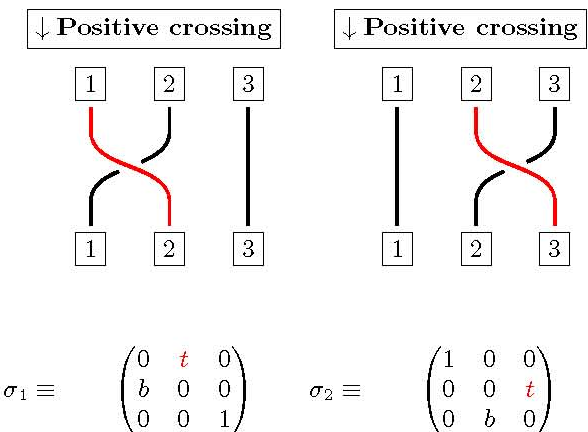}
\caption{Assigning matrices to basic positive crossings in $B_3$.}
\label{fig3}
\end{figure}

\begin{figure}[thbp]
\centering
\includegraphics[width=7cm,height=4cm]{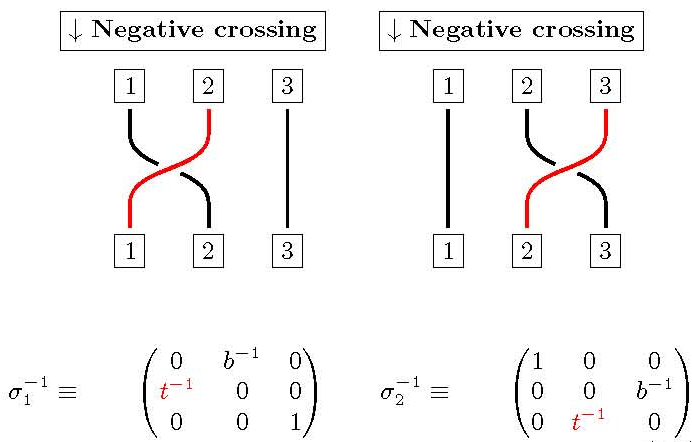}
\caption{Assigning matrices to basic negative crossings in $B_3$.}
\label{fig4}
\end{figure}

In examples below, we show how a typical braid in $B_3$ composed of some generator is depicted, and also how its matrix representation is calculated by multiplication of matrices representing its constituent generators:
	
	\begin{example}\label{sig2sig2in=identity} Illustrating $\sigma_2 \sigma^{-1}_2=I_3$, Figures \ref{fig5}.
	
	\begin{figure}[thbp]
\centering
\includegraphics[width=7cm,height=4cm]{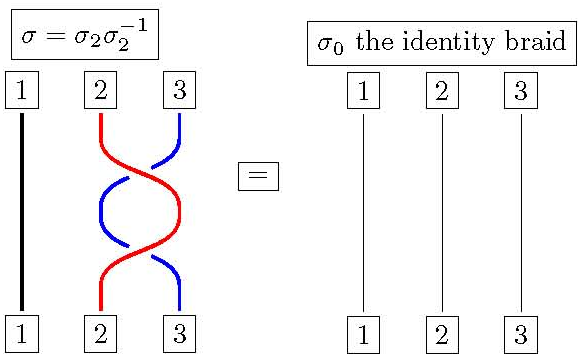}
\caption{Illustrating $\sigma_2 \sigma^{-1}_2=I_3$.}
\label{fig5}
\end{figure}

\begin{align*} \sigma &=\sigma_2 \sigma^{-1}_2  \equiv
				\begin{pmatrix} 1&0&0\\0&0&t\\0&b&0 \end{pmatrix}
				\begin{pmatrix} 1&0&0\\0&0&b^{-1}\\0&t^{-1}&0 \end{pmatrix}
				\\&=\begin{pmatrix} 1&0&0\\0&tt^{-1}&0\\0&0&bb^{-1} \end{pmatrix}=
				\begin{pmatrix} 1&0&0\\0&1&0\\0&0&1 \end{pmatrix}=I_3 \end{align*}
	\end{example}

	\begin{example}\label{sig2sig1invsig2} Illustrating $\sigma=\sigma_2 \sigma^{-1}_1 \sigma_2$, Figures \ref{fig6}.

\begin{align*}
\sigma &=\sigma_2 \sigma^{-1}_1 \sigma_2 \\&\equiv
			\begin{pmatrix} 1&0&0\\0&0&t\\0&b&0 \end{pmatrix}
			\begin{pmatrix} 0&b^{-1}&0\\t^{-1}&0&0\\0&0&1 \end{pmatrix}
			\begin{pmatrix} 1&0&0\\0&0&t\\0&b&0 \end{pmatrix}\\\nonumber\\
			&=\begin{pmatrix} 0&0&b^{-1}t\\0&tb&0\\bt^{-1}&0&0 \end{pmatrix}\nonumber
\end{align*}			
	\end{example}
	
\begin{figure}[thbp]
\centering
\includegraphics[width=4cm,height=7cm]{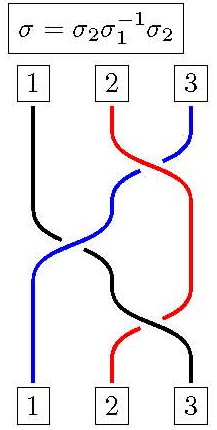}
\caption{Illustrating $\sigma=\sigma_2 \sigma^{-1}_1 \sigma_2$.}
\label{fig6}
\end{figure}

\begin{remark}\label{pathanalyzing1}{\bf[Path Analyzing Approach:]}\\
{\rm
{\bf (a)} It is very important to remark that, at each crossing where, for example, the strand $i$ crosses over the strand $i+1$, we in fact consider the parameter $\mathbf{t}$ as a {\it relative weight} given to strand $i$ relative to $i+1$,  hence the position $i\ i+1$ for $\mathbf{t}$ in the matrix. Similarly, the parameter $\mathbf{b}$ is a {\it relative weight} given to strand $i+1$ relative to $i$,  hence the position $i+1\ i$ for $\mathbf{b}$ in the matrix.
	
{\bf (b)} It useful to notice that, the matrix representation of a given braid $\sigma$ could be found in two ways. One way is by, formally, multiplying all matrices representing generators who make up the braid $\sigma$. Second way is, more intuitively, by direct construction of matrix entries, following the core idea of, $\mathbf{t}$ for {\bf top} and $\mathbf{b}$ for {\bf bottom} in positive crossings, and, $\mathbf{t^{-1}}$ for {\bf top} and $\mathbf{b^{-1}}$ for {\bf bottom} in negative crossings. This is done by analyzing the path of each strand from its starting higher point to its ending lower point, in the braid. Let us show this for the braid of Example \ref{sig2sig1invsig2}. The $1^{st}$ strand travels  downward from the higher $1$ all the way to the lower $3$. Doing so, this strand passes a negative crossing by being at the bottom, hence collecting a $b^{-1}$. Then it passes a positive crossing by being at the top, hence collecting a $t$. Multiplying these, gives the $13$-entry of the matrix representation as $b^{-1}t$. Similarly the $2^{nd}$ strand travels downward from the higher $2$ all the way to the lower $2$. Doing so, this strand passes a positive crossing by being at the top, hence a $t$, followed by another positive crossing by being at the bottom, hence a $b$. Thus the $22$-entry$=tb$. Similar "{\it{path analyzing}}" gives the $31$-entry$=bt^{-1}$.
}
\end{remark}
	
\subsection{\normalsize \textbf{The basic representation}}
Summarizing and putting the above idea in more rigorous terms, we have the following. At each crossing, where the strand $i$ crosses over the strand $i+1$ for example, we assign $\mathbf{t}$ to the {\bf t}op strand and $\mathbf{b}$ to the {\bf b}ottom strand. We consider the parameter $\mathbf{t}$ as a {\it relative weight} given to strand $i$ relative to $i+1$,  hence the position $i\ i+1$ for $\mathbf{t}$ in the matrix. Similarly, the parameter $\mathbf{b}$ is a {\it relative weight} given to strand $i+1$ relative to $i$,  hence the position $i+1\ i$ for $\mathbf{b}$ in the matrix (Figure \ref{fig2}). The same goes for the under-crossing situation. This leads us to the following theorem.
	
	\begin{theorem}\label{maintheorem}
		Let $B_n$ be the braid group with Artin generators $\sigma_i,\, 1\leq i \leq n-1$. Let $\mathbb{Z}[t^{\pm1}, b^{\pm 1}]$ be the ring of Laurent polynomials in two variables, over $\mathbb{Z}$. The following defines a representation of $B_n$ into $GL_n(\mathbb{Z}[t^{\pm1}, b^{\pm 1}])$.
		\begin{equation*} \Psi: B_n \rightarrow GL_n(\mathbb{Z}[t^{\pm1}, b^{\pm 1}])\end{equation*}
		\begin{align*} \Psi(\sigma_i)&:=I_{i-1} \oplus
		\begin{pmatrix} 0&t\\b&0 \end{pmatrix} \oplus I_{n-i-1}
		\\&=\begin{pmatrix}
		I_{i-1}
		& \vline &0 \cdots & \vline & 0 \cdots\\
		\hline
		0 \cdots & \vline &
		\begin{matrix}
		0 & t \\
		b & 0
		\end{matrix}
		& \vline & 0 \cdots \\
		\hline
		0 \cdots & \vline & 0 \cdots & \vline & I_{n-i-1}\\
		\end{pmatrix}  \end{align*}
		
		Here, $I_k$ is the identity matrix of size $k$. The above definition also implies that,
		\begin{equation*} \Psi(\sigma_i^{-1})= \begin{pmatrix}
		I_{i-1}
		& \vline &0 \cdots & \vline & 0 \cdots\\
		\hline
		0 \cdots & \vline &
		\begin{matrix}
		0 & b^{-1} \\
		t^{-1} & 0
		\end{matrix}
		& \vline & 0 \cdots \\
		\hline
		0 \cdots & \vline & 0 \cdots & \vline & I_{n-i-1}\\
		\end{pmatrix}  \end{equation*}
		
	\end{theorem}

	\begin{proof}
		
	 Being a representation follows easily from the definition of $\Psi$ and from direct and simple calculations on block matrices. However, here we use the path analyzing method. We only prove the cubic identity $\Psi(\sigma_{i}) \, \Psi(\sigma_{i+1}) \, \Psi(\sigma_{i})=\Psi(\sigma_{i+1}) \, \Psi(\sigma_{i}) \, \Psi(\sigma_{i+1})$. The remaining identities could be proved in a similar manner. We use the following diagrams in Figures \ref{fig7} and \ref{fig8}. Let us refer to the resulting matrix representation as $\psi$. As it is clear from both diagrams in Figures \ref{fig7} and \ref{fig8}, the only involved strands are ${i}^{th}$, ${(i+1)}^{th}$ and ${(i+2)}^{th}$ strands. \\
		
		In the diagram of $\sigma_{i}\, \sigma_{i+1} \,\sigma_{i}$ in Figure \ref{fig7},
		
\begin{figure}[thbp]
\centering
\includegraphics[width=7cm,height=5cm]{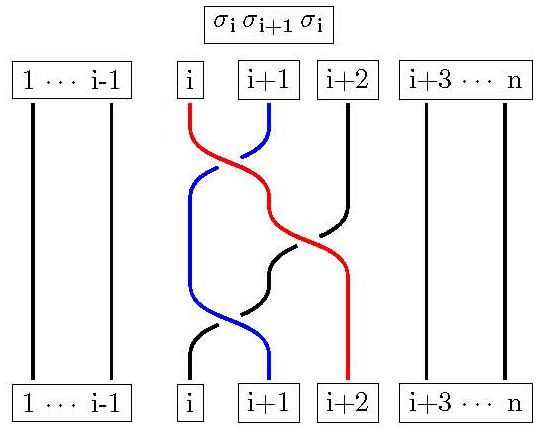}
\caption{Illustrating $\sigma_{i}\, \sigma_{i+1} \,\sigma_{i}$.}
\label{fig7}
\end{figure}	
		
we see that the ${i}^{th}$-strand moving downward from the higher $i$ to the lower $i+2$, passing through two positive crossings, both at the top (hence a $t$ followed by another $t$). Thus the $i\,i+2$-entry is equal to $t^2$, i.e., $\psi_{i \,i+2}=t^2$.
		
		The ${(i+1)}^{th}$-strand moving downward from the higher $i+1$ to the lower $i+1$, passing through a positive crossing at the bottom (hence a $b$), followed by another positive crossings at the top (hence a $t$). Thus the $i+1\,i+1$-entry is equal to $bt$, i.e., $\psi_{i+1\, i+1}=bt$.
		
		Finally, the ${(i+2)}^{th}$-strand moving downward from the higher $i+2$ to the lower $i$, passing through a positive crossing at the bottom (hence a $b$), followed by another positive crossings at the bottom (hence another $b$). Thus the $i+2\,i$-entry is equal to $b^2$, i.e., $\psi_{i+2 \, i}=b^2$.
		
		Moreover, since all other $k^{th}$-strands for $k \neq i,\,i+1,\, i+2$, go directly from the higher $k$ to the lower $k$ without any crossings, the diagonal entries $\psi_{k \, k}=1$ for $k \neq i,\,i+1,\, i+2$. All the other entries, not involved in any journey of any strand, are zero.    \newline
		
		Next, in the diagram of $\sigma_{i+1}\, \sigma_{i}, \,\sigma_{i+1}$ in Figure \ref{fig8},
		
\begin{figure}[thbp]
\centering
\includegraphics[width=7cm,height=5cm]{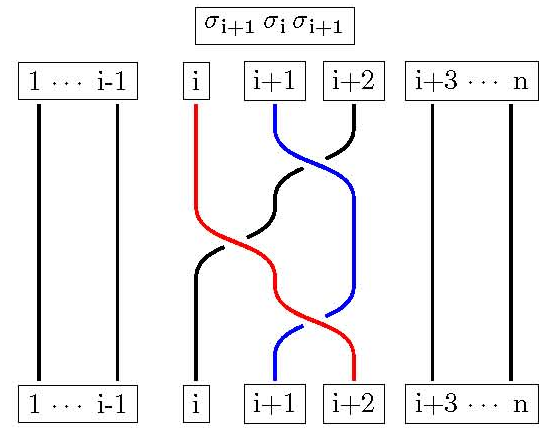}
\caption{Illustrating $\sigma_{i+1}\, \sigma_{i}, \,\sigma_{i+1}$.}
\label{fig8}
\end{figure}	

we see exact similar patterns, as in $\sigma_{i}\, \sigma_{i+1} \,\sigma_{i}$, for the ${i}^{th}$-strand and ${(i+2)}^{th}$-strand. For the ${(i+1)}^{th}$-strand the only difference is the {\bf order} of being at the top of a positive crossing followed by being at the bottom of another positive crossing (i.e. $tb$ instead of $bt$). Thus we get the similar results, $\psi_{i \,i+2}=t^2$, $\psi_{i+1\, i+1}=bt$, and $\psi_{i+2 \, i}=b^2$. Obviously the same is true for all other entries.
		Therefore we have the desired result,
		
			\begin{align*} 	\Psi(\sigma_{i})& \, \Psi(\sigma_{i+1}) \, \Psi(\sigma_{i})=\Psi(\sigma_{i+1}) \, \Psi(\sigma_{i}) \, \Psi(\sigma_{i+1})
			\\&= \begin{pmatrix}
			I_{i-1}
			& \vline &0 \cdots & \vline & 0 \cdots\\
			\hline
			0 \cdots & \vline &
			\begin{matrix}
			0 & 0&t^2 \\
			0&tb&0\\
			b^2&0 & 0
			\end{matrix}
			& \vline & 0 \cdots \\
			\hline
			0 \cdots & \vline & 0 \cdots & \vline & I_{n-i-2}\\
			\end{pmatrix} \end{align*} . %\newline
		
		This finishes the proof of the theorem.
	\end{proof}

\subsection{\normalsize \textbf{This representation is not faithfull}} \label{rep.not.faithful}
	
	Now by providing some examples, we will show that the kernel of the representation defined in Theorem \ref{maintheorem} is not trivial. Hence the representation is not faithful. In what follows, by abusing the notation, we simply use  $\sigma$ not only to denote a braid but also to denote its matrix representation $\Psi(\sigma)$ as well.
	
	\begin{example}\label{killer3}
		In $B_3$, for braids $\alpha$ and $\beta$ as shown in Figure \ref{fig9},
		
\begin{figure}[thbp]
\centering
\includegraphics[width=7cm,height=5cm]{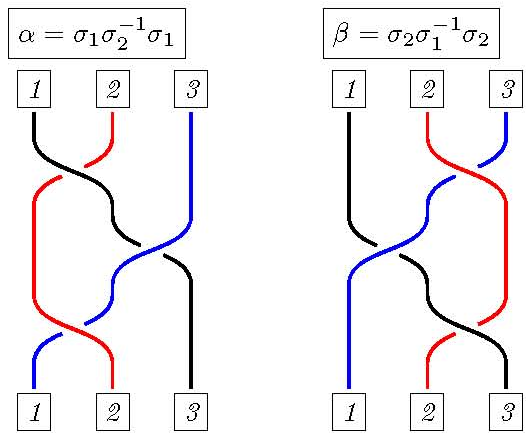}
\caption{Illustrating $\alpha$ and $\beta$.}
\label{fig9}
\end{figure}	

we will get equal matrix representations (from Theorem \ref{maintheorem}), i.e.,
		
		\begin{equation*} \alpha=\beta=\begin{pmatrix} 0&0&b^{-1}t\\0&tb&0\\bt^{-1}&0&0 \end{pmatrix} \end{equation*}
		
		Thus the representation of $\alpha\beta^{-1}$ is equal to identity, i.e.,
		
		\begin{equation*} \alpha\beta^{-1}= I_3=\begin{pmatrix} 1&0&0\\0&1&0\\0&0&1 \end{pmatrix}\end{equation*}
		
		But $\alpha\beta^{-1}$ as a braid  in $B_3$, is not equal to the trivial braid $\sigma_0$, as shown in Figure \ref{fig10}.
		
\begin{figure}[thbp]
\centering
\includegraphics[width=7cm,height=6cm]{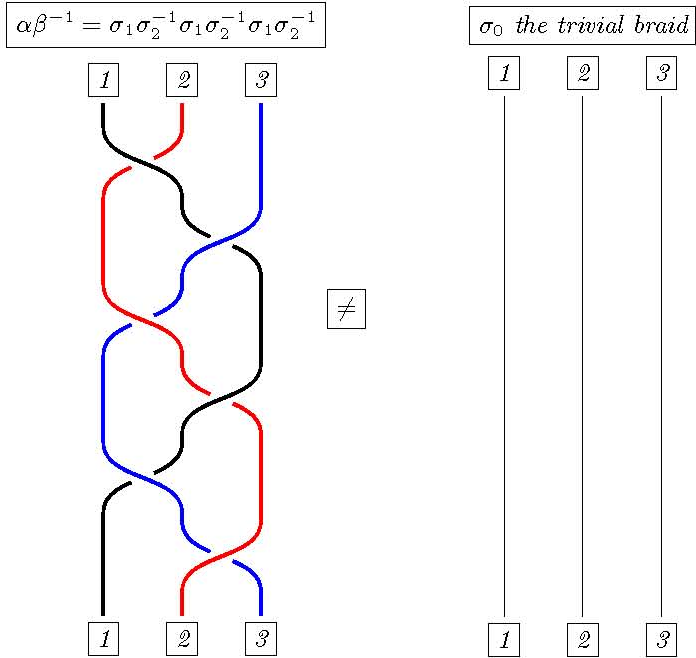}
\caption{Illustrating $\alpha\beta^{-1}\neq I_3$.}
\label{fig10}
\end{figure}		

Therefore we have a non trivial braid, namely, $$\alpha \beta^{-1}=\sigma_1 \sigma^{-1}_2 \sigma_1 \sigma^{-1}_2 \sigma_1 \sigma^{-1}_2$$ in the kernel of this representation.
	\end{example}
	
	For more examples one can verify that, for example in $B_3$, if we let, $\alpha=\sigma_1^{-1} \sigma_2^{2} \sigma_1^{-1} \sigma_2^{-1} \sigma_1^{2} \sigma_2^{-1} $ and
	$\beta=\sigma_2^{-1} \sigma_1^{2} \sigma_2^{-1} \sigma_1^{-1} \sigma_2^{2} \sigma_1^{-1} $, then $\alpha \beta^{-1}=I_3$ from the representation. But it is not hard to see that as a braid in $B_3$, $\alpha \beta^{-1} \neq \sigma_0$.
	
	Also in $B_4$, One can check that if we let, $\sigma=\sigma_1^{2}\sigma_3^{2}\sigma_2\sigma_3^{-2}\sigma_1^{-2}\sigma_2^{-1}$ then the representation of $\sigma$ is the identity matrix $I_4$, but as a braid in $B_4$, $\sigma$ is not equal to the trivial braid $\sigma_0$.
	
In Section \ref{General Burau representation} we generalize and improve the representation defined in this section to the one with a much smaller kernel.

%%%%%%%%%%%%%%%%%%%%%%%%%
\section{\Large{A General Burau representation}} \label{General Burau representation}
In this section, by continuing the same intuitive path analyzing approach, only by introducing an additional parameter, we greatly improve the representation in Section \ref{core idea} into the one with much smaller kernel. This more general representation includes the unreduced Burau representation \cite{Burau}, as a special case.\\

For that, not only we assign $\mathbf{t}$ and $\mathbf{b}$ as described before, but also we assign to the top strand a new {\it weight}, $\mathbf{\omega}$, relative to itself (hence the  $i\ i$ position). After some inspection we realise the only way for the representation to work is if $\mathbf{\omega = 1-tb}$. This leads to the more general representation defined by,

\begin{equation*} \Psi(\sigma_i)=\begin{pmatrix}
I_{i-1}
& \vline &0 \cdots & \vline & 0 \cdots\\
\hline
0 \cdots & \vline &
\begin{matrix}
1-tb & t \\
b & 0
\end{matrix}
& \vline & 0 \cdots \\
\hline
0 \cdots & \vline & 0 \cdots & \vline & I_{n-i-1}\\
\end{pmatrix} , \end{equation*}
which, after renaming of the parameters, is {\bf Case 1} in the following theorem. Similarly, if we assign to the bottom strand the {\it weight} $\mathbf{1-tb}$, we will arrive at the {\bf Case 2} in the following theorem, after renaming of the parameters. Moreover, as we will see, we can not assign the new weight to both strands at the same time.\\

In the following theorem we have renamed the parameters to better exhibit the generality of the result. In other words, to present an answer to the following question. {\it What is the most general possible $2$-by-$2$ matrix that can be assigned to the basic crossing of two strands in $B_2$, which can then be extended to a representation of $B_n$ in a natural way}?

\begin{theorem}\label{maintheorem2}
Let $B_n$ be the braid group with Artin generators $\sigma_i,\, 1\leq i \leq n-1$. Let $\mathbb{Z}[a^{\pm1}, b^{\pm 1}, c^{\pm 1}, d^{\pm 1}]$ be the ring of Laurent polynomials in four variables, over $\mathbb{Z}$. The map,

\begin{align*} \Psi&: B_n \rightarrow GL_n(\mathbb{Z}[a^{\pm1}, b^{\pm 1}, c^{\pm 1}, d^{\pm 1}])
\\&\Psi(\sigma_i):=I_{i-1} \oplus
\begin{pmatrix} a&b\\c&d \end{pmatrix} \oplus I_{n-i-1}
\\& =\begin{pmatrix}
I_{i-1}
& \vline &0 \cdots & \vline & 0 \cdots\\
\hline
0 \cdots & \vline &
\begin{matrix}
a & b \\
c & d
\end{matrix}
& \vline & 0 \cdots \\
\hline
0 \cdots & \vline & 0 \cdots & \vline & I_{n-i-1}\\
\end{pmatrix} \end{align*}
will define a non-trivial representation of $B_n$ into $GL_n(\mathbb{Z}[a^{\pm1}, b^{\pm 1}, c^{\pm 1}, d^{\pm 1}])$, if and only if one of the following cases occurs:\newline

{\bf Case 1:} $d=0$, $c=\frac{1-a}{b}$ and $a\neq 1$, which results in the representation,

\begin{align*} \Psi&: B_n \rightarrow GL_n(\mathbb{Z}[a^{\pm1}, b^{\pm 1}])
\\&\Psi(\sigma_i):=I_{i-1} \oplus
\begin{pmatrix} a&b\\\frac{1-a}{b}&0 \end{pmatrix} \oplus I_{n-i-1}
\\&=\begin{pmatrix}
I_{i-1}
& \vline &0 \cdots & \vline & 0 \cdots\\
\hline
0 \cdots & \vline &
\begin{matrix}
a & b \\
\frac{1-a}{b} & 0
\end{matrix}
& \vline & 0 \cdots \\
\hline
0 \cdots & \vline & 0 \cdots & \vline & I_{n-i-1}\\
\end{pmatrix} \end{align*}

{\bf Case 2:} $a=0$, $c=\frac{1-d}{b}$ and $d \neq 1$, which results in the representation,

\begin{align*} \Psi&: B_n \rightarrow GL_n(\mathbb{Z}[d^{\pm1}, b^{\pm 1}])
\\&\Psi(\sigma_i):=I_{i-1} \oplus
\begin{pmatrix} 0&b\\\frac{1-d}{b}&d \end{pmatrix} \oplus I_{n-i-1}
\\&=\begin{pmatrix}
I_{i-1}
& \vline &0 \cdots & \vline & 0 \cdots\\
\hline
0 \cdots & \vline &
\begin{matrix}
0 & b \\
\frac{1-d}{b} & d
\end{matrix}
& \vline & 0 \cdots \\
\hline
0 \cdots & \vline & 0 \cdots & \vline & I_{n-i-1}\\
\end{pmatrix} \end{align*}

{\bf Case 3:} $a=d=0$ and $b\neq0$, $c\neq0$, which results in the representation,

\begin{align*} \Psi&: B_n \rightarrow GL_n(\mathbb{Z}[c^{\pm1}, b^{\pm 1}])
\\&\Psi(\sigma_i):=I_{i-1} \oplus
\begin{pmatrix} 0&b\\c&0 \end{pmatrix} \oplus I_{n-i-1}
\\&=\begin{pmatrix}
I_{i-1}
& \vline &0 \cdots & \vline & 0 \cdots\\
\hline
0 \cdots & \vline &
\begin{matrix}
0 & b \\
c & 0
\end{matrix}
& \vline & 0 \cdots \\
\hline
0 \cdots & \vline & 0 \cdots & \vline & I_{n-i-1}\\
\end{pmatrix} \end{align*}

In all the above formulas, $I_k$ is the identity matrix of size $k$ and $1\leq i \leq n-1$.
\end{theorem}

\begin{proof}
We will have a non-trivial representation if and only if all matrices representing $\sigma_i$'s are non-singular (invertible) and satisfy the cubic and commuting relations in \eqref{cubic relations} and \eqref{commuting relations}. For non-singularity we will always make sure that $ad-bc\neq0$ in all cases. It is clear the commuting relations satisfy under no conditions. Therefore, we only need to satisfy the cubic relations.\newline

Without loss of generality, it is enough to show the results in $B_3$, where for the qubic relation to hold we need,

\begin{equation*} \Psi(\sigma_1)\Psi(\sigma_2)\Psi(\sigma_1) - \Psi(\sigma_2)\Psi(\sigma_1)\Psi(\sigma_2) =0\end{equation*}

A simple calculation shows that,

\begin{equation*}\Psi(\sigma_1)\Psi(\sigma_2)\Psi(\sigma_1) - \Psi(\sigma_2)\Psi(\sigma_1)\Psi(\sigma_2) =\end{equation*}
\begin{equation*} \left(\begin{array}{rrr}
a b c + a^{2} - a & a b d & 0 \\
a c d & -a^{2} d + a d^{2} & -a b d \\
0 & -a c d & -b c d - d^{2} + d
\end{array}\right) \end{equation*}

For this matrix to be a zero matrix, it is clear that at least one of the $a, b, c, d$ must be zero. If, for example, we assume only $d=0$ then all the entries of the above matrix are zero except the first entry of the diagonal. From which we get the relation $c=\frac{1-a}{b}$. This is case 1. Case 2 and case 3 follow as easy and with a similar analysis. \\

We finish the proof by remarking that, if $b=0$ ($c=0$), it implies $c=0$ ($b=0$) and $a=d=1$, which results in the trivial representation. Also, all other possible cases will result in singular (not invertible) matrices.
\end{proof}

\begin{remark}
{\rm In Theorem \ref{maintheorem2}, Case 2 seems essentially similar to Case 1. Focusing on Case 1, we would like to point out that, as a special case, if we let $1-a=b$ we will arrive at,

\begin{align*} \Psi(\sigma_i)&:=I_{i-1} \oplus
\begin{pmatrix} 1-b&b\\1&0 \end{pmatrix} \oplus I_{n-i-1}
\\&=\begin{pmatrix}
I_{i-1}
& \vline &0 \cdots & \vline & 0 \cdots\\
\hline
0 \cdots & \vline &
\begin{matrix}
1-b & b \\
1 & 0
\end{matrix}
& \vline & 0 \cdots \\
\hline
0 \cdots & \vline & 0 \cdots & \vline & I_{n-i-1}\\
\end{pmatrix} \end{align*}
This is indeed the unreduced Burau representation. On the other hand Case 3 is exactly the simple representation we defined in Section \ref{core idea} (with only names of variables changed).}
\end{remark}

\section{\Large{Concluding remarks}} \label{Concluding remarks}
In this paper, we have used a new intuitive and graphical path analyzing approach to define a simple representation of dimension $n$ for the braid group $B_n$. Next, continuing the same approach, we have extended that representation to a more general one which also includes the unreduced Burau representation.

The current work is an updated and upgraded version of an earlier version on arXiv \cite{arasharxivev2}. Professors, Vladimir Shpilrain\cite{Shpilrain}, Valeriy Bardakov\cite{Bardakov} and Inna Sysoeva\cite{Sysoeva1}, sent me their valuable and constructive feedback and comments on that arXiv version. In fact, part of those communications encouraged the author to improve and generalize the representation in Section \ref{core idea} to the one in Section \ref{General Burau representation}. Moreover, It was pointed out to me by Valeriy Bardakov that similar results have been found in\cite{Mikhalchishina}, as a result of studying the local linear representations of the braid group $B_3$ which led to the homogeneous local representations of $B_n$ for $n\geq 2$. It was also pointed out to me by Inna Sysoeva that similar results first appeared in the paper\cite{Tong etal}, as part of the study of $(n+m-2)$ representations of $B_n$, for $m \geq 2$. The results of the present paper have been achieved completely independently and using a new method, namely, the "path analyzing method". This approach has the advantage that it applies a very simple and intuitive method capturing the fundamental interactions of the strands in a braid, i.e., over-crossings, under-crossings or no-crossings.\\

The novelty and contribution of the current work, in addition to the main results, is our new approach. In this approach we intuitively follow each strand in a braid and create a {\it history} for the strand as it interacts with other strands via over-crossings, under-crossings or no-crossings. This, directly, leads us to the desired representations, as follows. At each crossing, where the strand $i$ crosses over the strand $i+1$ for example, we assign $\mathbf{t}$ to the {\bf t}op strand and $\mathbf{b}$ to the {\bf b}ottom strand. We consider the parameter $\mathbf{t}$ as a {\it relative weight} given to strand $i$ relative to $i+1$,  hence the position $i\ i+1$ for $\mathbf{t}$ in the matrix. Similarly, the parameter $\mathbf{b}$ is a {\it relative weight} given to strand $i+1$ relative to $i$,  hence the position $i+1\ i$ for $\mathbf{b}$ in the matrix. Thus, we arrive at the very simple representation of Theorem \ref{maintheorem} defined by,

\begin{equation*} \Psi(\sigma_i)=\begin{pmatrix}
I_{i-1}
& \vline &0 \cdots & \vline & 0 \cdots\\
\hline
0 \cdots & \vline &
\begin{matrix}
0 & t \\
b & 0
\end{matrix}
& \vline & 0 \cdots \\
\hline
0 \cdots & \vline & 0 \cdots & \vline & I_{n-i-1}\\
\end{pmatrix} ,\end{equation*}
which is also {\bf Case 3} in Theorem \ref{maintheorem2}, after renaming the parameters.\\

Next, to improve this representation to the one with much smaller kernel we go one step further, as follows. Not only we assign $\mathbf{t}$ and $\mathbf{b}$ as described above, but also we assign to the top strand a new {\it weight}, $\mathbf{\omega}$, relative to itself (hence the  $i\ i$ position). It turns out that the only possible value for such a new weight is $\mathbf{\omega = 1-tb}$. This leads to the more general representation defined by,

\begin{equation*} \Psi(\sigma_i)=\begin{pmatrix}
I_{i-1}
& \vline &0 \cdots & \vline & 0 \cdots\\
\hline
0 \cdots & \vline &
\begin{matrix}
1-tb & t \\
b & 0
\end{matrix}
& \vline & 0 \cdots \\
\hline
0 \cdots & \vline & 0 \cdots & \vline & I_{n-i-1}\\
\end{pmatrix} ,\end{equation*}
which, after renaming of the parameters, is {\bf Case 1} in Theorem \ref{maintheorem2}. Similarly, if we assign to the bottom strand the weight $\mathbf{1-tb}$ ( relative to itself), we will have

\begin{equation*} \Psi(\sigma_i)=\begin{pmatrix}
I_{i-1}
& \vline &0 \cdots & \vline & 0 \cdots\\
\hline
0 \cdots & \vline &
\begin{matrix}
0 & t \\
b & 1-tb
\end{matrix}
& \vline & 0 \cdots \\
\hline
0 \cdots & \vline & 0 \cdots & \vline & I_{n-i-1}\\
\end{pmatrix} ,\end{equation*}
which is {\bf Case 2} in Theorem \ref{maintheorem2}, after renaming of the parameters. From the proof of Theorem \ref{maintheorem2}, it is also clear that, despite our wish, it is not possible to assign such a new weight to both strands at the same time. \\

However, we have not ruled out the possibility of giving a greater number of {\it relative} weights to each strand, not only relative to the one it interacts at the moment of crossing. But also relative to other strands. This avenue is worth to explore, as it could potentially lead to some new representations! This is a work in progress.

\section*{Acknowledgement}
First and foremost, the author would like to thank peer reviewers for their invaluable recommendations, which helped to improve presentation of this paper, greatly. The author would also like to thank Professors Vladimir Shpilrain, Valeriy Bardakov and Inna Sysoeva for their kind attention to this work and for their instructive communications with me. Last but not least, I would like to thank Elena Ryzhova for fruitful conversations.%\newline

\noindent\hrulefill

%%%%%%%%%%%%%%%%%%%%%%%%%
\renewcommand\refname{REFERENCES}
	
\end{document}